\documentclass[10pt]{amsart}
\usepackage{amsmath,amscd,latexsym,verbatim,amssymb}
\usepackage{times}

\newcommand{\resp}{{\it resp.} }

\newcommand{\ie}{{\it i.e.} }
\newcommand{\eg}{{\it e.g.} }

\newcommand{\Q}{\mathbf{Q}}

\newcommand{\C}{\mathbf{C}}

\newcommand{\Z}{\mathbf{Z}}

\renewcommand{\epsilon}{\varepsilon}
\renewcommand{\phi}{\varphi}

\renewcommand{\lim}{\varprojlim}

\newcounter{spec}
{\end{list}}

\swapnumbers

\newtheorem{thm}{Theorem}[subsection]
\newtheorem{lemma}[thm]{Lemma}
\newtheorem{prop}[thm]{Proposition}
\newtheorem{cor}[thm]{Corollary}

\theoremstyle{definition}

\numberwithin{equation}{section}

 at10pt

\renewcommand{\qed}{\hfill $\square$\medskip}

\begin{document}
 \title{A note on 1-motives.}
\author{Yves Andr\'e}
\address{Institut de Math\'ematiques de
Jussieu\\4 place Jussieu\\75005
Paris\\France.}
\email{yves.andre@imj-prg.fr}
 \date{\today}
 
 \maketitle

 \begin{sloppypar}

  \begin{abstract} We prove that for $1$-motives defined over an algebraically closed subfield $k$ of $\C$, viewed as Nori motives, the motivic Galois group coincides with the Mumford-Tate group. In particular, the Hodge realization of the tannakian category of Nori motives generated by $1$-motives is fully faithful. 
  
  This result extends an earlier result by the author, according to which Hodge cycles on abelian $k$-varieties are motivated (a weak form of the Hodge conjecture). \end{abstract}

  \bigskip
  \bigskip
 \section{Introduction}
 
\subsection{} The study of points on semiabelian varieties  (\ie extensions of abelian varieties by tori) is a very classical topic of diophantine geometry.
   In algebraic geometry, it also played a crucial role in the guise of Deligne's $1$-motives \cite{D}. Over an algebraically closed subfield $k$ of $\C$, a $1$-motive $[\mathcal L \to \mathcal G]$ is given by a morphism   from a lattice $\mathcal L$ to a semiabelian variety $\mathcal G$ (taking a basis of the lattice, this amounts to the data of a finite number of points on $\mathcal G$). This notion served as a double test ground:
  
   $i)$ for Deligne's theory of mixed Hodge structures if $k=\C$ ($1$-motives form an easily describable  full subcategory of the category of mixed Hodge structures), 
   
   $ii)$ for Grothendieck's dream of mixed motives ($1$-motives are those coming from varieties of dimension $\leq 1$, whence the name ``$1$-motive").  

\smallskip  Nowadays, a well-defined {\it tannakian} category $MM(k)$ of mixed motives with rational coefficients over a field $k\subset \C$ is available in full generality, in two different (independent, but canonically equivalent) versions due to M. Nori \cite{N} and J. Ayoub \cite{Ay1} respectively (see \cite{A3} for a survey). Nori's construction is more elementary and puts in light the universality property of motives, while the geometric origin of morphisms is more apparent in Ayoub's version which is constructed out of Voevodsky's triangulated category.
 Anyway, one knows how to associate unconditionally a {\it motivic Galois group} to any motive over $k$. 
 
 One can attach an object of $MM(k)$ not only to any $k$-variety, but also to $1$-motives over $k$. We denote by $MM(k)_1$ (\resp  $MM(k)_1^{\otimes}$ ) the full subcategory (\resp tannakian subcategory) of $MM(k)$ generated by $1$-motives: objects of $MM(k)_1^{\otimes}$ are constructed from those of $MM(k)_1$ by saturating under tensor products, duals, subquotients.  

\subsection{} The tannakian category $MM(k)$ admits a fiber functor (the {\it Hodge realization}) toward the tannakian category $MHS$ of mixed Hodge structures. By tannakian duality, this provides an injective homomorphism between the Mumford-Tate group of any motive (\ie the tannakian group attached to its Hodge realization) and its motivic Galois group. 

Let us assume that $k$ is {\it algebraically closed}. A version of the Hodge conjecture (the {\it Hodge-Nori conjecture} \cite{Ar}) then predicts that the Hodge realization is full; a slightly more precise version, in terms of tannakian groups, predicts that the motivic Galois group equals the Mumford-Tate group. This is what we prove in this paper in the special case of $1$-motives:

\begin{thm}\label{th} The Hodge realization $  \,MM(k)_1^{\otimes} \to MHS$ is fully faithful, and identifies $MM(k)_1^{\otimes} $ with a tannakian subcategory\footnote{in particular, it is stable under subobjects taken in $MHS$.} of $MHS$.  

A fortiori, the motivic Galois group of any $1$-motive over $k$ coincides with its Mumford-Tate group. 
\end{thm}

 This confers a genuine motivic content to the description of Mumford-Tate groups of $1$-motives presented in \cite{Be}, and in particular to the notion of deficiency \cite{Ber}. This could also shed some light on P. Jossen's work on the Mumford-Tate conjecture \cite{J} and on several recent works on periods of $1$-motives (see \eg \cite{H-W}) in their relation to the Grothendieck period conjecture (and to our generalization of Grothendieck's conjecture to a non necessarily algebraic ground field $k$ \cite[23.4.1]{A2}).

 \subsection{} 
 In \cite{A1}, we proved that the motivic Galois group of any abelian $k$-variety coincides with its Mumford-Tate group. In that setting, motivic Galois groups were understood in the context of the tannakian category of pure motives $M(k)$ defined in terms of motivated correspondences. According to \cite{Ar}, $M(k)$ is canonically equivalent to the socle of $MM(k)$ (\ie its full subcategory of semisimple objects), which allows to interpret our theorem in \cite{A1} as a confirmation of the Hodge-Nori conjecture for abelian varieties, and Theorem \ref{th} as an extension of it.
In fact, Theorem \ref{th} has the following consequence:
 
 \begin{cor} The tannakian subcategory of $MM(k)_1^{\otimes}$ consisting of semisimple objects is canonically equivalent to the tannakian subcategory of $M(k)$ generated by the motives of abelian varieties. \end{cor}

\subsection{} In the prehistory of the theory of motives, one was limited to morphisms of systems of realizations (a.k.a. absolute Hodge correspondences) instead of morphisms of ``geometric origin" as should be genuine motivic morphisms, in some way. In that weaker context, Deligne proved that the absolute Hodge tannakian group attached to any complex abelian variety coincides with its Mumford-Tate group, and J.-L. Brylinski extended this result to $1$-motives. Our result enhances Brylinski's result to the genuine motivic context, with a completely parallel argument, namely: 

$i)$ we replace Deligne's result by the stronger result that any Hodge cycle on a complex abelian variety is motivated \cite{A1}, translated in terms of Nori motives via \cite{Ar},

$ii)$ we mimick Brylinski's deformation argument, using to a motivic version of the ``theorem of the fixed part" due to Nori, Jossen and Ayoub (independently; since Nori's and Jossen's notes do not seem easily accessible, we rely on Ayoub's published version \cite{Ay1} and the compatibility with Nori's framework \cite{C-GAS}).
   
 \smallskip The progress between Brylinski's theorem and the theorem of this paper is thus a shadow of the progress of the theory of motives in the last 35 years, and can be restated as follows: {\it for $1$-motives, tensor Hodge classes do not only satisfy the expected compatibilities between various realizations, they indeed ``come from geometry"} (in a non-naive sense, more apparent in Ayoub's setting).

  \section{} Let us begin with some preliminaries about $1$-motives and Nori motives.
 As above, let $k$ be an algebraically closed subfield of $\C$  and $MM(k)$ denote the tannakian category of Nori motives over $k$ with rational coefficients \cite{N}\cite{H-MS} (see also \cite{BV-H-P} for a new viewpoint on the tensor structure). The Betti realization provides a fiber functor $R_B: MM(k)\to Vec_\Q$, which is canonically enriched as a  
 fiber functor $R_H: MM(k) \to MHS$ toward the tannakian category of rational mixed Hodge structures.

  There is also a category of effective (Nori) mixed motives $MM^{eff}(k)$, from which $MM(k)$ is constructed by formally inverting the Lefschetz motive. It is not known whether the faithful functor $MM^{eff}(k)\to MM(k)$ is full. 
  
  Let $DM(k)_1$ be the abelian category of Deligne $1$-motives up to isogeny. In  \cite[6.1]{Ay-BV}, it is shown that $DM(k)_1$ is canonically equivalent to a full abelian subcategory of $MM^{eff}(k)$: this is the thick abelian subcategory generated by motives of the form $h_1(X,Y)$ and the unit motive $\Q(0)$. We denote by $MM(k)_1$ its essential image in $MM(k)$.  According to \cite[6.9]{Ay-BV}, the composed functor $$DM(k)_1\to MM^{eff}(k)\to MM(k) \to MHS$$ coincides with the (rational) Hodge realization of $1$-motives constructed by Deligne \cite{D}.  
  
  \begin{prop} This composed functor is fully faithful. A fortiori $DM(k)_1\to MM(k)_1$ is an equivalence.
  \end{prop} 
  
  \begin{proof} Deligne actually proved that $DM(k)_1\to MHS$ is fully faithful in the case $k=\C$. The case of an algebraically closed subfield $k$ follows. Indeed let $M_i =[\mathcal L _i\to \mathcal G_i], i= 1, 2,$ be $1$-motives over $k$, each given by a lattice $\mathcal L_i$ and a morphism from $\mathcal L_i$ to  a semi-abelian variety $\mathcal G_i$, extension of an abelian variety $A_i$ by a torus $T_i$. It suffices to show that any morphism  $M_{1\C}  \to  M_{2\C}$ descends to $k$, \ie that the morphism $\mathcal G_{1\C} {\to} \mathcal G_{2\C}$ descends to $k$. By Cartier duality, this amounts to the well-known fact that the induced morphism of abelian varieties $A^\vee_{2\C} \stackrel{f}{\to} A^\vee_{1\C}$ descends to $k$. 
  
  The second statement follows from the first since all involved functors are faithful.
   \end{proof} 
  
   In particular $MM(k)_1$ is abelian. Reminding that the socle of $MM(k)$ is canonically equivalent to the category $M(k)$ of pure motives constructed in \cite{A1} (\cite[6.4]{Ar}, see also \cite[10.2]{H-MS}), we also deduce that any semisimple object of $MM(k)_1$ is isomorphic to a direct sum in $M(k)$ of the motive $h_1(A)$ of an abelian variety and copies of $\Q(0), \Q(1)$.

  \section{} Let $M =[\mathcal L \to \mathcal G]$ be a $1$-motive over $k$, given by a morphism from a lattice $\mathcal L$ to  a semi-abelian variety $\mathcal G$, extension of an abelian variety $A$ by a torus $T$. Up to replacing $M$ by the direct sum of $M$ and its Cartier dual, which changes neither the motivic Galois group, nor the Mumford-Tate group, we may assume that $M$ is symmetric (= polarizable) in the sense of \cite{Br}.
  
     By \cite[0.6.2]{A1} and by the identifications indicated above, the theorem holds for the tannakian subcategory of {\it semisimple} $1$-motives (up to isogeny), in particular for the tannakian subcategory generated by the $1$-motive $M_0 := Gr_W M  = [\mathcal L \stackrel{0}{\to} A\times T]$. Note that the image of $M_0$ in $MM(k)_1$ is the semisimplification of the image of $M$.  Let $P$ be the Mumford-Tate group of $M$, and let us fix a polarization of $M$ (hence of $A$).  
  
  \begin{lemma} Polarized $1$-motives $N$ with $Gr_W N =  M_0$ and Mumford-Tate group contained in $P$ fit into an algebraic family $\mathcal M$ parametrized by a smooth connected $k$-variety $X$.  \end{lemma}
  
  See \cite[2.2.8.6]{Br} (and also \cite[1.8]{J}).  The $1$-motive $M$ (\resp $M_0$) is a fiber of $\mathcal M$ at a $k$-point $x$ (\resp $x_0$) of $X$. 
 The ``mixed Shimura variety"  $X$ is just a torus bundle over an abelian variety, analytically isomorphic to $W_{-1}P(\Z)\backslash W_{-1}P(\C)/ (F^0\cap W_{-1}P(\C))$. In particular the monodromy of the family at $x_1$ is given by the natural action of $W_{-1}P(\Z)$ on $H(M)$ and its Zariski hull is the connected group $W_{-1}P$. Any Hodge (\ie $P$-invariant) tensor is thus invariant under monodromy.  The point $x$ is ``Hodge-generic" in the family in the sense that the Mumford-Tate group of $\mathcal M_x= M$ is maximal, equal to $P$.

Let $L$ be a $P$-stable line in some mixed tensor construction $T^\bullet R_B(M)$ over $R_B(M)$ (with Tate twists).  By (Tate) twisting, one reduces to the case where $L$ is $P$-invariant, \ie generated by a Hodge tensor. We have to show that $L$ is the realization of unit submotive in $T^\bullet M$, knowing that its parallel transport to $x_0$ is the realization of a unit submotive in $T^\bullet M_0$.

   \section{} Let $Y$ be a smooth connected $k$-variety. Let $\mathcal N \in MM(Y)$ be a motivic local system, viewed as a mixed motive over $k(Y)$ unramified over $Y$. Its Betti realization is a local system $R_B(\mathcal N)$ of $\Q$-vector spaces on $Y(\C)$.  Taking the fiber at a point $y\in Y(k)$, one gets in this way a fiber functor $R_{B,y}: MM(Y) \to Vec_\Q$, which is then enriched as a fiber functor $R_{mon,y}: MM(Y) \to Rep_\Q \,\pi_1(Y(\C), y)$ (monodromy realization). 
     Taking tannakian duals, one gets a morphism $G_{mon}(\mathcal N, y)\to G_{mot}(\mathcal N, y)$ (in fact an embedding of closed subgroups of $GL(R_{B,y}(\mathcal N)$), where $G_{mon}(\mathcal N, y)$ is the algebraic monodromy group attached to $R_{mon, y}(\mathcal N)$.

  \begin{prop} $G_{mon}(\mathcal N, y)$ is a normal subgroup of $ G_{mot}(\mathcal N, y)$. If $R_B(\mathcal N)$ is constant, then $\mathcal N$ is constant, \ie is the pull-back of a motive in $MM(k)$.  
  \end{prop} 
  
 \begin{proof} See \cite[th. 40, rem. 41]{Ay2}\footnote{in this reference, Ayoub uses a complex geometric generic point of $Y$ rather than $y$, but functor fibers become isomorphic as usual.}; 
  the proof is given in \cite[2.57]{Ay1} in the context of Ayoub's category of mixed motives, which by \cite{C-GAS} is equivalent to the category of Nori motives.  The result also appears in unpublished works by Nori and by Jossen (in the context of Nori motives properly). \end{proof}
   
\medskip {\it Application}:  let $\mathcal M \in MM(X)$ be the motivic local system attached to the family of $1$-motives of the lemma. Let $\mathcal N \in MM(X)$ correspond to the representation of $ G_{mot}(\mathcal M, x)$ generated by $L$ inside $T^\bullet R_{B,x}(\mathcal M) = T^\bullet R_{B }( M)$. Because $L$ is fixed by $ G_{mon}(\mathcal N, x)$, it follows from the first part of the proposition that $R_B(\mathcal N)$ is a constant local system. By the second part, $\mathcal N$ itself is constant. Since $R_{B, x_0}(\mathcal N)$ contains the parallel transport of $L$ at $x_0$ which is the realization of a unit submotive in $\mathcal N_{x_0} \subset T^\bullet M_0$, we conclude that $L$ is the Betti realization of a unit submotive in $\mathcal N_{x} \subset T^\bullet M$ (which coincides a posteriori with $\mathcal N_x$ itself). This proves Theorem \ref{th}. \qed

\smallskip One may wonder\footnote{as Brylinski already did in his absolute Hodge context \cite[end of 2.2]{Br}.} whether there is a more direct alternative argument by devissage (with respect to the weight) rather than by deformation, in order to perform the reduction to the case of abelian varieties.

 \end{sloppypar}

\begin{thebibliography}{I}
   
  \bibitem{A1} Y. ANDR\'E, {\it Pour une th\'eorie inconditionnelle des motifs},  Publ. Math. I.H.E.S.  {\bf 83}  (1996), 5-49. 
  
   
  
  \bibitem{A2} $-$, {Une introduction aux motifs}, Panoramas et Synth\`eses {\bf 17} SMF (2004).

 
  \bibitem{A3} $-$, {\it Groupes de Galois motiviques et p\'eriodes}, S\'eminaire Bourbaki, {Novembre 2015}, 
  S. F. M. Ast\'erisque (2017), 
 
 
 \bibitem{Ar} D. ARAPURA, {\it  An abelian category of motivic sheaves}, Advances in
Math. {\bf 233}  1 (2013), 135-195.
 
 \bibitem{Ay1} J. AYOUB, {\it L'alg\`ebre de Hopf et le groupe de Galois motiviques
d'un corps de caract\'eristique nulle}, Journal f\"ur
die reine und angew. Mathematik {\bf 693} (2014), partie I: 1-149 ; partie II: 151-226.
  
 \bibitem{Ay2} $-$,  {\it Periods and the conjectures of Grothendieck and Kontsevich-Zagier}, EMS Newsletter {\bf 91} mars 2014, 12-18.
 
\bibitem{Ay-BV} $-$, L. BARBIERI-VIALE, {\it Nori 1-motives},  Math. Annalen {\bf 361}, n$^\circ$ 1-2 (2015), 367-402.

\bibitem{BV-H-P} L. BARBIERI-VIALE, A. HUBER, M. PREST, {\it Tensor structure for Nori motives}, arXiv:1803.00809
 
\bibitem{Be} C. BERTOLIN, {\it Le radical unipotent du groupe de Galois motivique d'un $1$-motif}, Math. Annalen {\bf 327} 3 (2003), 585-607.


\bibitem{Ber} D. BERTRAND,  {\it Special points and Poincar\'e biextensions}, with an appendix by B. Edixhoven, arXiv:1104.5178

 \bibitem{Br} J.-L. BRYLINSKI, {\it 1-motifs et formes automorphes}, in Journ\'ees automorphes, Publ. Math. Univ. Paris VII, 15 (1983), 43-106.

  \bibitem{C-GAS} U. CHOUDHURY, M. GALLAUER ALVES de SOUZA, {\it An isomorphism of motivic Galois groups},  arXiv:1410:6104. 

 \bibitem{D} P. DELIGNE, {\it Th\'eorie de Hodge III}, Publ. Math. I.H.E.S. {\bf 44} (1974), 5-77.

 \bibitem{H-MS} A. HUBER, S. M\"ULLER-STACH, {Periods and Nori motives} (with contributions of B. Friedrich and J. von Wangenheim), Springer Ergebnisse {\bf 65}, 2017.
 
  \bibitem{H-W} A. HUBER, G. W\"USTHOLZ, {\it Periods of $1$-motives},  arXiv:1805.10104
 
 \bibitem{J} P. JOSSEN, {\it On the Mumford-Tate conjecture for 1-motives}, Inventiones Math {\bf 195} 2 (2014), 393-439.

\bibitem{N} M. NORI, Course at Tata Inst. Fund. Res., notes by N. Fakhruddin, Mumbai (2000).

 
 \end{thebibliography}
 \end{document}